\documentclass[12pt]{article}

\usepackage[inner=32mm,outer=31mm,tmargin=30mm,bmargin=40mm]{geometry}
\usepackage[ngerman,english]{babel}

\usepackage{latexsym,amsfonts,amsmath,amssymb,epsfig,tabularx,amsthm,dsfont,mathrsfs}

\usepackage{graphicx}
\usepackage{enumerate}

\usepackage{booktabs}
\usepackage{titling}
\newcommand{\subtitle}[1]{%
  \posttitle{%
    \par\end{center}
    \begin{center}\large#1\end{center}
    \vskip0.5em}%
}

\usepackage{hyperref} 
\hypersetup{colorlinks=true,
        linkcolor=black,
        citecolor=black,
        urlcolor=blue}

\theoremstyle{plain}

\newtheorem{theorem}{Theorem}[section]
\newtheorem{lemma}[theorem]{Lemma}

\newtheorem{corollary}[theorem]{Corollary}

\theoremstyle{definition}
\newtheorem{remark}[theorem]{Remark}

\renewcommand{\P}{\operatorname{\mathbb P}}
\newcommand{\expect}{\operatorname{\mathbb{E}}}

\DeclareMathOperator{\Uniform}{unif}

\newcommand{\iid}{\stackrel{\textup{iid}}{\sim}}

\newcommand{\ind}{\mathds{1}}

\DeclareMathOperator{\median}{med}

\DeclareMathOperator{\card}{card}
\newcommand{\ran}{\textup{ran}}

\newcommand{\nonada}{\textup{nonada}}
\newcommand{\ada}{\textup{ada}}

\newcommand{\eps}{\varepsilon}
\newcommand{\embed}{\hookrightarrow}

\renewcommand{\vec}{\boldsymbol}

\newcommand{\R}{{\mathbb R}}

\newcommand{\N}{{\mathbb N}}

\DeclareMathAlphabet{\mathup}{OT1}{\familydefault}{m}{n}

\DeclareMathOperator{\shrink}{Shrink}

\DeclareMathOperator{\spot}{Spot}
\DeclareMathOperator{\select}{BucketSelect}

\newcommand{\wt}{\widetilde}
\newcommand{\widebar}[1]{\mbox{\kern1.5pt\hbox{\vbox{\hrule height 0.6pt \kern0.35ex
        \hbox{\kern-0.15em \ensuremath{#1 }\kern0.0em}}}}\kern-0.1pt}

\newlength{\fixboxwidth}
\usepackage{soul}

\usepackage{latexsym,amsfonts,amsmath,amssymb,mathrsfs}
\usepackage{url}
\usepackage{graphicx}
\usepackage{color}
\usepackage{euscript}
\usepackage{verbatim}
\usepackage{hyperref}

\definecolor{darkgreen}{rgb}{0,0.5,0}

\begin{document}

\title{Uniform approximation of vectors using adaptive randomized information}

\author{Robert J. Kunsch\thanks{RWTH Aachen University,
at the Chair of Mathematics of Information Processing,
Pontdriesch~10,
52062 Aachen, Email: kunsch@mathc.rwth-aachen.de
}, 
Marcin Wnuk\thanks{Institut für Mathematik, 
Osnabrück University, Albrechtstraße 28a, 49076 Osnabrück, 
Email: marcin.wnuk@uni-osnabrueck.de, corresponding author}}

\date{\today}

\maketitle
\begin{abstract}
We study approximation of the embedding $\ell_p^m \embed \ell_{\infty}^m$, $1 \leq p \leq 2$,
based on randomized adaptive algorithms that use arbitrary linear 
functionals as information on a problem instance.
We show upper bounds for which the complexity $n$  
exhibits only a $(\log\log m)$-dependence. 
Our results for $p=1$ 
lead to an example of a gap of order $n$ (up to logarithmic factors)
for the error between best adaptive and non-adaptive Monte Carlo methods.
This is the largest possible gap for linear problems.
\end{abstract}


\section{Introduction}

Let $S \colon F \to G$ be a linear operator between Banach spaces $F$ and $G$.
Approximating $S$ by a randomized algorithm $A_n$ means that, for any given input $f \in F$, $A_n$ yields a $G$-valued random variable $A_n(f)$.
Its accuracy can be measured in terms of the expected error with respect to the $G$-norm.
The error should be reasonably small not just for an individual input~$f$ but for any input from the unit ball of $F$.
The usual way to define the Monte Carlo error of $A_n$ is
\begin{equation} \label{eq:MCerr}
    e(A_n,S) := \sup_{\|f\|_F \leq 1} \expect \|A_n(f) - S(f)\|_G \,.
\end{equation}
Typical algorithms for linear problems are homogeneous~\cite{KK24}, 
that is, 
for any scalar $t \in \R$ we have $A_n(t f) = t A_n(f)$,
and the algorithms presented in this paper are no exception.
Homogeneity implies $\expect \|A_n(f) - S(f)\|_G \leq \|f\|_F \, \cdot \, e(A_n,S)$
which justifies defining the error by the supremum over inputs from the unit ball in~$F$.
In \emph{information-based complexity} (IBC)~\cite{TWW88},
the main bottleneck for an algorithm~$A_n$ is
that it must generate an output based on incomplete information,
in this paper we allow up to $n$ evaluations $L_i(f)$ of linear functionals $L_i \in F'$,
so-called \emph{measurements}.
The number~$n$ is what we call the \emph{cardinality} $\card(A_n)$.
\emph{Non-adaptive} randomized algorithms $A_n$ exhibit a clear structure:
\begin{enumerate}
    \item Draw a list of $n$ functionals $(L_1,\ldots,L_n) \in (F')^n$
        according to a chosen probability distribution on $(F')^n$.
    \item Take measurements $\vec{y} = (y_1, \ldots, y_n) = N(f) := (L_1(f),\ldots,L_n(f))$ to obtain information on the input~$f$.
    \item Compute an output $\phi(\vec{y}) \in G$,
        where also the knowledge on the choice of the information map~$N$ may be taken into account
        and even additional random numbers can be used.
\end{enumerate}
In contrast, for \emph{adaptive} algorithms,
the first and the second stage are intertwined:
For $i = 2,\ldots,n$, the functional $L_i$ may be drawn while taking into account previously measured information $y_1,\ldots,y_{i-1}$.
In the quest for optimal algorithms we define \emph{minimal errors}
\begin{equation*}
    e^{\ran}(n,S) := \inf_{A_n} e(A_n,S)
\end{equation*}
where the infimum is take over all (possibly adaptive) randomized algorithms $A_n$ that use at most $n$ linear measurements. If we restrict the infimum to non-adaptive algorithms,
we write $e^{\ran,\nonada}(n,S)$.
One may also restrict to deterministic \mbox{methods}, i.e., algorithms that do \emph{not} use random numbers, but here we focus on randomized methods.
Studying lower bounds for these abstract error notions~\cite{H92,KNW24,Ma91}
requires a more formal description of general algorithms.
In the context of upper bounds, however, it only matters to give a clear description and analysis of specific methods.

In this paper we describe and analyse adaptive algorithms
for approximating the identical embedding on $\R^m$ with respect to different norms,
\begin{equation*}
    S \colon \ell_p^m \to \ell_q^m, \quad \vec{x} \mapsto \vec{x} \,,
\end{equation*}
where for $1 \leq p < \infty$, $\ell_p^m$ is the vector space $\R^m$ equipped with the $\ell_p$-norm
\begin{equation*}
    \|\vec{x}\|_p := \left(|x_1|^p + \ldots + |x_m|^p\right)^{1/p} \,,
\end{equation*}
and for the target space we mainly consider the case of $q=\infty$ with the uniform norm
\begin{equation*}
    \|\vec{x}\|_{\infty} := \max\{|x_1|,\ldots,|x_m|\} \,.
\end{equation*}
We abbreviate $S$ by the shorthand $\ell_p^m \embed \ell_q^m$.
These finite-dimensional sequence space embeddings are important for understanding approximation of more complex infinite-dimensional operators such as Sobolev embeddings~\cite{H92,Ma91}
and approximation in general, a short survey about some classical results in these sequence spaces is contained in the introduction of~\cite{KNW24}.

One fundamental question in numerical analysis is whether adaptive methods are more powerful than non-adaptive methods~\cite{No96}. 
This question crucially depends on the assumptions we make:
Linearity of the solution operator $S$ or weaker properties, the geometry of the input set (in our case the convex and symmetric unit ball of~$F$), the error criterion, the class of admissible information functionals, and whether we consider deterministic or randomized algorithms.
There even exist numerical problems that can only be solved using adaptive algorithms~\cite{KNR19}.
For linear problems (with convex and symmetric input sets),
however, it is well known that in the deterministic setting adaption essentially does not help~\cite[Sec~2.7]{TW80}.
In the randomized setting it was only recently established that adaption \emph{does} make a difference for some linear problems:
Heinrich~\cite{He24,He22,He23b,He23} studied adaption for several problems in 
a setting where admissible information consists of function evaluations, in~\cite{He23} he gave an example where the gap between the error of adaptive and non-adaptive methods could reach a factor of order $n^{1/2}$ (up to logarithmic terms).
In a setting where arbitrary linear functionals can be evaluated for information,
a similar gap of order $n^{1/2}$ was shown in~\cite{KNW24},
namely for sequence space approximations $\ell_1^m \embed \ell_2^m$.
It remained open how big such a gap can be.
The current paper is a direct continuation of~\cite{KNW24}.
By establishing upper bounds for the approximation error of $\ell_1^m \embed \ell_\infty^m$ using adaptive algorithms, see Theorem~\ref{thm:errorrates},
in combination with lower bounds for the same problem and non-adaptive algorithms~\cite[Thm~2.7]{KNW24},
we find that the gap 
can be of order $n$, up to logarithmic terms,
see Theorem~\ref{thm:adagap} and Corollary~\ref{cor:diagop}.
In a recent paper~\cite{KNU24} it is found that a gap of this size is as large as it can get.

The adaptive algorithm we present in this paper is based on several ideas that have been developed by Woodruff and different co-authors~\cite{IPW11,LNW17,LNW18}.
They study the slightly different problem of \emph{stable sparse recovery}
where a recovery scheme with output $\vec{x}^\ast$ for approximating any given vector $\vec{x} \in \R^m$ is said to provide an \emph{$\ell_p/\ell_q$ guarantee} for $k$-sparse approximation
if, for some $\epsilon > 0$, the bound
\begin{equation*}
    \|\vec{x} - \vec{x}^\ast\|_p \leq (1+\epsilon) \min_{\text{$k$-sparse }\vec{x}'} \|\vec{x} - \vec{x}'\|_q
\end{equation*}
holds with high probability. Schemes for this problem can be used to tackle sequence space embeddings,
for instance, in \cite[Thm~3.1]{KNW24} an adaptive $\ell_2/\ell_2$ guarantee has been used for solving the linear problem $\ell_1^m \embed \ell_2^m$.
We might as well use an $\ell_\infty/\ell_2$ guarantee~\cite[Thm~3 \& Sec~3]{NSWZ18} to establish upper bounds for the problem $\ell_p^m \embed \ell_\infty^m$.
In this paper, however, we give a comprehensive presentation of all methods required for $\ell_p^m \embed \ell_\infty^m$ and, in order to keep the presentation self-contained, include all essential proofs
with explicit estimates.
The fact that sparse recovery has slightly different demands than the linear problem we are interested in,
allows for certain simplifications in the analysis.
In an upcoming paper on $\ell_q$-approximation for $q < \infty$ it even turns out that the algorithm itself can be simplified~\cite{KW24c}.
On the other hand, we are interested in the full regime of $1 \leq p < q \leq \infty$,
which adds another layer to the analysis.

The paper is structured as follows:
In Section~\ref{sec:blocks} we review all essential parts of the algorithm one-by-one.
In Section~\ref{sec:shrink} we describe the adaptive core procedure which provides the speed-up leading to a $\log\log m$-dependence of the cardinality (rather than a $\log m$ dependence known from non-adaptive methods).
In Section~\ref{sec:AdaAppInf} we combine everything to describe the algorithm and analyse its cost
which is finally translated into the error rates of our main result Theorem~\ref{thm:errorrates}.
In Section~\ref{sec:adagap} we apply these bounds to particular examples to demonstrate the potential superiority of adaptive methods over non-adaptive algorithms.

\subsection*{Asymptotic notation}

We use asymptotic notation to compare functions $f$ and $g$ that depend on variables $(\eps,\delta,m)$ or $(n,m)$, writing \mbox{$f \preceq g$} if there exists a constant \mbox{$C > 0$} such that $f \leq C g$ holds for ``large~$m$ and small~$\eps,\delta$'' (say, for $m \geq 16$ and \mbox{$0 < \eps,\delta < \frac{1}{2}$}), 
or for ``large $m$ and $n$'', respectively. \emph{Weak asymptotic equivalence} $f \asymp g$ means $f \preceq g \preceq f$. We also use \emph{strong asymptotic equivalence} $f \simeq g$ to state $f/g \to 1$ for $\eps,\delta \to 0$.
The implicit constant~$C$ or the convergence $f/g$ is to be understood for fixed $p$.

\section{Toolkit for adaptive approximation}
\label{sec:blocks}

For $m \in \N$ let $[m] := \{1,\ldots,m\}$,
for $\vec{x} = (x_i)_{i\in [m]} \in \R^m = \R^{[m]}$ and $K \subseteq [m]$ define the sub-vector $\vec{x}_K := (x_i)_{i \in K} \in \R^K$.
The projection $\vec{z} = \vec{x}^\ast_K \in \R^m$ onto the coordinates $K$ is defined by
\begin{equation*}
    z_i := \begin{cases}
        x_i &\text{for } i \in K\,, \\
        0 &\text{else.}
    \end{cases}
\end{equation*}
The output of the algorithm will be a vector $\vec{x}^\ast_K$ where $K$ is a set of presumably important coordinates (i.e. coordinates $i$ for which $|x_i|$ is large) that has been identified adaptively. In particular, the algorithm uses information from direct evaluations of the entries $x_i$ for $i \in K$ once the set $K$ has been fixed.
The core of the algorithm is the adaptive identification of the most important coordinate $j \in B$
from a given so-called \emph{bucket} $B \subseteq [m]$, see Section~\ref{sec:shrink}.
In preparation, Section~\ref{sec:hash} studies the process of randomly splitting the domain $[m]$ into buckets such that 
with high probability each bucket contains at most one important coordinate and, moreover, within each such bucket we can identify a heavy coordinate with the required probability.
Only a small portion of buckets will contain an important coordinate,
selecting these relevant buckets is the topic of Section~\ref{sec:select}.
All stages of the algorithm are subject to failure with small probability,
throughout this section we give different names ($\alpha,\alpha_k,\delta_0,\delta_1,\delta_2$)
for these individual failure probabilities
to facilitate the description of the whole algorithm later in Section~\ref{sec:AdaAppInf}.

\subsection{Hashing}
\label{sec:hash}

Hashing is a well established methodology in computer science with use in data storage and retrieval. We are not interested in the specific requirements of such applications and ignore aspects of implementation. Instead, we only use the basic idea of identifying chunks of the index set~$[m]$ by a random hash value. In effect, we split a vector~$\vec{x}$ into sub-vectors, thus reducing the big problem of approximating~$\vec{x}$ into a collection of smaller and easier approximation problems.

Let $D \in \N$ and $\vec{H} = (H_i)_{i=1}^m$ be a family of pairwise independent random variables uniformly distributed on $[D]$, that is, $H_i \sim \Uniform[D]$.
We call $\vec{H}$ a \emph{hash function} or \emph{hash vector}, its entries are called \emph{hash values}.
This defines disjoint buckets
\begin{equation} \label{eq:J_d}
    J_d = J_d^{\vec{H}} := \{i \in [m] \colon H_i = d\} \,, \quad d \in [D]\,.
\end{equation}
For $j \in [m]$ we denote by
\begin{equation*}
    B_j = B_j^{\vec{H}} := J_{H_j}^{\vec{H}} = \{i \in [m] \colon H_i = H_j\}
\end{equation*}
the bucket that contains $j$.
The main benefit from hashing is that for a fixed heavy coordinate $x_j$,
with a proper choice of the parameter $D$,
the entries of $\vec{x}_{B_j \setminus \{j\}}$ are likely to have
small absolute values,
as the following result shows.

\begin{lemma} \label{lem:hash}
    Let $\vec{x} \in \R^m$, $j \in [m]$, $1 \leq p < \infty$, and $\alpha \in (0,1)$.
    Then we have the probabilistic bound
    \begin{equation*}
        \P\left(\|\vec{x}_{B_j \setminus \{j\}}\|_{p}
                > \frac{\lVert \vec{x}_{[m] \setminus \{j\}  } \rVert_p}{(\alpha D)^{1/p}}\right)
            \leq \alpha \,.
    \end{equation*}
\end{lemma}
\begin{proof}
As in \cite[Lem~3.2]{IPW11}, we start by computing
the $p$-moment of $\|\vec{x}_{B_j}\|_p$
where,
thanks to the pairwise independence of the entries of the hash vector~$\vec{H}$,
for $i \in [m]$ with $i\neq j$ we have $\P(H_i = H_j) = \frac{1}{D}$:
\begin{align*}
    \expect\left[\|\vec{x}_{B_j \setminus \{j\}}\|_{p}^p\right]
        &= \sum_{i \in [m]\setminus\{j\}} \P(H_i = H_j) \cdot |x_i|^p
        = \frac{1}{D} \cdot \|\vec{x}_{[m] \setminus \{j\}}\|_p^p \,.
\end{align*}
The probabilistic bound
then follows from Markov's inequality.
\end{proof}

\begin{remark} \label{rem:Lem1optimal}
    It might be surprising that Markov's inequality in the proof of Lemma~\ref{lem:hash}
    is already the best what we can do and more sophisticated concentration inequalities
    for fully independent entries of the hash vector $\vec{H}$ do not help.
    Indeed, let $r := \lfloor \alpha D \rfloor$ and choose a subset $I \subseteq [m]\setminus\{j\}$ with $\# I = r$. Consider the vector~$\vec{x}$ with non-zero entries
    $x_i= r^{-1/p}$ for $i \in I$, and $x_i = 0$ for $i \in [m] \setminus (I\cup \{j\})$,
    such that $\|\vec{x}_{[m]\setminus\{j\}}\|_p = 1$.
    Assuming $H_i \iid \Uniform[D]$ we find
    \begin{align*}
        \P\left( \|\vec{x}_{B_j\setminus\{j\}}\|_p \geq \frac{\lVert \vec{x}_{[m] \setminus \{j\}  } \rVert_p}{(\alpha D)^{1/p}} \right)
            &\geq \P\left( \|\vec{x}_{B_j\setminus\{j\}}\|_p \geq r^{-1/p} \right) \\
            &= 1 - \prod_{i \in I} \P\bigl( H_i \neq H_j\bigr)
            = 1 - \left(1 - \frac{1}{D}\right)^r \\
            &\approx \frac{r}{D} \approx \alpha
            \qquad \text{for $D \gg r \gg 1$.}
    \end{align*}
    Since Lemma~\ref{lem:hash} must also hold for fully independent hash values,
    it is asymptotically optimal.
\end{remark}

In Section~\ref{sec:shrink} we describe a method that can identify an important coordinate~$j$ of a vector~$\vec{x}$
if it fulfils a so-called \emph{heavy-hitter condition} on its bucket~$B_j$ with an appropriate constant $\gamma > 1$:
\begin{equation} \label{eq:heavy-hitter-cond-general}
    \|\vec{x}_{B_j \setminus\{j\}}\|_2
        \leq \frac{|x_j|}{\gamma} \,.
\end{equation}
The left-hand side of~\eqref{eq:heavy-hitter-cond-general} can be
viewed 
as $\ell_2$-noise hampering the detection of~$x_j$.
The $\ell_2$-norm occurs naturally in the stochastic analysis of random measurements on vectors which causes problems for~$p>2$.
In fact, it turns out that hashing does not really help in the case of~$p>2$,
see Remark~\ref{rem:p>2nobuckets}.
Therefore, from now on all main results are only given for $1 \leq p \leq 2$.
Uniform approximation requires that hashing isolates \emph{all} heavy coordinates~$j$ above a certain threshold $|x_j| \geq \eps > 0$ in the sense of~\eqref{eq:heavy-hitter-cond-general}.

\begin{corollary}\label{cor:hash}
    Let $1 \leq p \leq 2$ and $\vec{x} \in \R^m$ with $\|\vec{x}\|_p \leq 1$.
    Further let $\gamma > 1$, and $\eps,\delta_0 \in (0,1)$.
    Define the set of $\eps$-heavy coordinates
    \begin{equation*}
        I := \left\{j \in [m] \colon |x_j| \geq \eps \right\}
        \qquad\text{where}\quad
        \# I \leq k_0
            := \left\lfloor \eps^{-p} \right\rfloor .
    \end{equation*}
    If we take
    \begin{equation*}
        D := \left\lceil \left(\frac{\gamma 
                    }{\eps}\right)^p \cdot \frac{k_0}{ \delta_0 } \right\rceil
    \end{equation*}
    and draw a hash vector $\vec{H} \in [D]^m$
    with pairwise independent entries $H_j \sim \Uniform[D]$,
    then
    \begin{equation*}
        \P\left(\forall j \in I \colon 
        \left\|\vec{x}_{B_j\setminus\{j\}}\right\|_2
            \leq \frac{|x_j|}{\gamma} \right)
            \geq 1 - \delta_0 \,.
    \end{equation*}
\end{corollary}
\begin{proof}
    The hashing parameter $D$ is chosen such that
    $\|\vec{x}\|_p/(\alpha D)^{1/p} \leq \eps/\gamma$
    with $\alpha = \delta_0/k_0$, hence,
    for a fixed $\eps$-heavy coordinate $j \in I$,
    Lemma~\ref{lem:hash}  implies
    \begin{equation*}
        \P\left(\left\|\vec{x}_{B_j\setminus\{j\}}\right\|_2
                    > \frac{|x_j|}{\gamma}
            \right)
        \leq \P\left(\left\|\vec{x}_{B_j\setminus\{j\}}\right\|_p
                        > \frac{\eps}{\gamma}
                \right)
            \leq \frac{\delta_0 }{k_0} \,,
    \end{equation*}
    where we exploited $\|\cdot\|_2 \leq \|\cdot\|_p$ for norms in $\R^m$ for $p \leq 2$.
    Employing a union bound for all $j \in I$, we arrive at the desired assertion.
\end{proof}

\begin{remark}
    By an example similar to Remark~\ref{rem:Lem1optimal}
    we see that, essentially, Corollary~\ref{cor:hash} cannot be improved.
    Consider a vector with $k_1$ entries of value $\eps := (2k_1)^{-1/p}$
    and $k_2 = \lfloor 2k_1 \cdot (\gamma/2)^p\rfloor$ entries of value $\eps/\gamma$, all other entries put to~$0$.
    This leads to a vector with $\ell_p$-norm at most~$1$.
    For large $D$ and i.i.d.\ hashing, the probability of finding at least one bucket with heavy coordinate $x_j=\eps$ and $\|\vec{x}_{B_j\setminus\{j\}}\|_2 > \eps/\gamma$ will be of asymptotic order $k_1^2 \gamma^p/D \asymp (\gamma/\eps^2)^p / D \asymp \delta_0$
    for $D \gg k_1 \gg 1$
    and $1 \leq p \leq 2$.
\end{remark}


In subsequent sections we do not always need the randomness of the hashing in our arguments
and thus write the lower case $\vec{h} = (h_i)_{i=1}^m$ for a fixed realization of the random vector~$\vec{H}$.
This also means that all random parameters of subordinate routines
are considered independent of $\vec{H}$.

\subsection{Bucket selection}
\label{sec:select}

In the context of stable sparse recovery with $\ell_\infty/\ell_2$ guarantee \cite[Sec~3]{NSWZ18},
an efficient algorithm for simultaneously estimating the $\ell_2$-norms of sub-vectors $\vec{x}_{J_d}$ for a large number of buckets $J_d$ is used.
In the literature on streaming algorithms this method is commonly known as \emph{partition count sketch}, 
see its description in the proof of~\cite[Lem~54]{LNW17}.
We are not that much interested in the individual bucket norm estimates
but rather use them as scores to find the buckets with the largest norm.
The respective routine shall thus be called \emph{bucket selection}.
It returns a set $I_k \subseteq [D]$, with $\# I_k = k$, such that (with high confidence) we have $d \in I_k$ for all buckets $J_d$ that contain a heavy hitter $j \in J_d$ in the sense of
\begin{equation} \label{eq:heavybucket}
    |x_j| \geq \eps \quad\text{and}\quad \|\vec{x}_{J_d \setminus\{j\}}\|_2 \leq \frac{|x_j|}{\gamma} \,,
\end{equation}
where $\gamma > 1$ is a suitable constant and $k$ needs to be tuned for $\eps > 0$.

Assume that we have a \emph{partition} of the domain $[m]$ into buckets $J_d = J_d^{\vec{h}}$
specified by a hash vector $\vec{h} \in [D]^m$.
Fix algorithmic parameters $R,G,k \in \N$, where $R$ is odd.
Generate independent hash vectors $\vec{H}^{(1)},\ldots,\vec{H}^{(R)} \iid \Uniform[G]^D$
and draw Rademacher coefficients $\sigma_{ri} \iid \Uniform\{\pm 1\}$, 
for \mbox{$r \in [R]$} and \mbox{$i \in [m]$}.
We then perform $n_1 := R \cdot G$ random measurements
\begin{equation*} 
    Y_{r,g} = L_{r,g}(\vec{x})
        := \sum_{d \in [D] \colon H^{(r)}_d = g} \, \sum_{i \in J_d} \sigma_{ri} \cdot x_i
    \qquad r \in [R],\; g \in [G].
\end{equation*}
(A collection of linear measurements like this is called a \emph{sketch} of $\vec{x}$.)
That way we execute $R$~repetitions of a grouped measurement where for each repetition the buckets are randomly sorted into $G$ groups and on each group a Rademacher functional is evaluated.
Hence,
for each bucket $J_d$ there are exactly $R$ measurements
\begin{equation*}
    \hat{Y}_{r,d} := Y_{r,H_d^{(r)}}, \qquad r \in [R] \,,
\end{equation*}
that are influenced by $\vec{x}_{J_d}$.
Using these values, for each $d \in [D]$ we compute a bucket score
by taking the median of absolute values,
\begin{equation*}
    Z_d := \median \left\{|\hat{Y}_{r,d}| \colon r \in [R]\right\} ,
\end{equation*}
which is uniquely defined as $R$ is odd.
Based on these scores we find an index set $I_{k} \subseteq [D]$ with cardinality $\# I_{k} = k$ such that
\begin{equation*}
    \min\{Z_d \colon d \in I_{k}\} \geq \max\{Z_d \colon d \in [D] \setminus I_{k}\} \,.
\end{equation*}
This set is the output of the bucket selection algorithm,
\begin{equation} \label{eq:bucketselect}
    \select_{R,G,k}^{\vec{h}}(\vec{x}) := I_{k} \,.
\end{equation}

\begin{lemma} \label{lem:bucketselect}
    Let $1 \leq p \leq 2$ and $\vec{x} \in \R^m$ with $\|\vec{x}\|_p \leq 1$.
    Given a hash vector $\vec{h} \in [D]^m$
    with corresponding buckets $J_d = J_d^{\vec{h}} \subseteq [m]$,
    consider the index set of important buckets that we aim to identify:
    \begin{equation*}
        Q := \left\{d \in [D] \colon \exists j \in J_d \text{ with } |x_j| \geq \eps \text{ and } \|\vec{x}_{J_d \setminus\{j\}}\|_2 \leq \frac{|x_j|}{8\sqrt{2}} \right\} . 
    \end{equation*}
    If for $\delta_1 \in (0,1)$
    we take an odd integer
    $R \geq 2 \log_2 \frac{D}{2\delta_1}$
    and
    \begin{equation*}
        k := \left\lfloor \left(\frac{8\sqrt{2} 
                                    }{\eps}\right)^p \right\rfloor,
        \qquad
        G := 4k \,,
    \end{equation*}
    then we have
    \begin{equation*}
        \P\left(Q \subseteq \select_{R,G,k}^{\vec{h}}(\vec{x})\right)
            \geq 1 - \delta_1 \,. 
    \end{equation*}
\end{lemma}
\begin{proof}
    Define $\hat{k} := \lfloor k / 4 \rfloor$ and
    let $\hat{Q} \subset [D]$ be a (not necessarily unique) $\hat{k}$-element index set of the $\hat{k}$~buckets with the largest $\ell_2$-norm, i.e.
    \begin{equation*}
        \# \hat{Q} = \hat{k} \qquad\text{and}\qquad
        \min_{d \in \hat{Q}} \|\vec{x}_{J_d}\|_2
            \geq \max_{d \in [D]\setminus \hat{Q}} \|\vec{x}_{J_d}\|_2 \,.
    \end{equation*}
    A well-known result on best $k$-term approximation, see e.g.~\cite{CDD09}, applied to the vector $(\|\vec{x}_{J_d}\|_2)_{d\in[D]}$ of bucket $\ell_2$-norms, provides the estimate
    \begin{equation} \label{eq:best-k-bucket}
        \sqrt{\sum_{d \in [D]\setminus \hat{Q}} \|\vec{x}_{J_d}\|_{2}^2 }
            \leq \left(\hat{k}+1\right)^{-\left(\frac{1}{p} - \frac{1}{2}\right)} 
                    \underbrace{\left(\sum_{d=1}^{D} \|\vec{x}_{J_d}\|_{2}^p \right)^{\frac{1}{p}}}_{
                        \leq \|\vec{x}\|_p \leq 1
                    }
            \leq 2 k^{-\left(\frac{1}{p} - \frac{1}{2}\right)} .
    \end{equation}
    For a fixed $d \in [D]$, the measurement $\hat{Y}_{r,d}$
    can be viewed as a sum of uncorrelated random variables $\ind\left[H_{h_i}^{(r)} = H_d^{(r)}\right] \cdot \sigma_{ri} \cdot x_i$, for $i \in [m]$.
    The indicator functions define the set of all companion buckets with which $J_d$ is grouped together in the $r$-th repetition:
    \begin{equation*}
        C_d^{(r)} := \left\{d' \in [D]\setminus\{d\} \colon H_{d'}^{(r)} = H_d^{(r)} \right\} .
    \end{equation*}
    Membership random variables $\ind[d' \in C_d^{(r)}]$, $d' \in [D] \setminus \{d\}$, are i.i.d.\ Bernoulli random variables with success probability $\frac{1}{G}$.
    A union bound
    for $d' \in \hat{Q} \setminus \{d\}$
    thus gives
    \begin{equation} \label{eq:hit-S_k} 
        \P\left(\hat{Q} \cap C_d^{(r)} \neq \emptyset
            \right)
            \leq \frac{\hat{k}}{G} \leq \frac{1}{16}
    \end{equation}
    where we used that $\#\hat{Q} \setminus\{d\} \leq\hat{k}$.
    Conditioning on the complementary event, we find the following second moment:
    \begin{align}
        \expect \left[\hat{Y}_{r,d}^2 \;\middle|\; \hat{Q} \cap C_d^{(r)} = \emptyset\right]
            &= \sum_{i=1}^m \P\left(h_i \in C_d^{(r)} \cup \{d\} \right) \cdot \underbrace{\sigma_{ri}^2}_{=1} \cdot \, x_i^2 
                \nonumber\\
            &= \|\vec{x}_{J_d}\|_2^2
            + \frac{1}{G} \sum_{d \in [D]\setminus \hat{Q}} \|\vec{x}_{J_d}\|_{2}^2 
                \label{eq:sketchmoment}\\
            &\stackrel{\eqref{eq:best-k-bucket}}{\leq} \|\vec{x}_{J_d}\|_2^2
                + \frac{4 k^{1-2/p}}{G} 
                \nonumber\\
            &= \|\vec{x}_{J_d}\|_2^2 + k^{-2/p}
                \nonumber\\
            &\leq \|\vec{x}_{J_d}\|_2^2
                + \frac{\eps^2}{128}  
                \,. 
                \nonumber
    \end{align}
    If $\|\vec{x}_{J_d}\|_2 \leq \frac{\eps}{8\sqrt{2}}$,
    then this conditional expectation is bounded by $\eps^2/64$ and
    Chebyshev's inequality gives
    \begin{equation*}
        \P\left(|\hat{Y}_{r,d}| \geq \frac{\eps}{2} \;\middle|\;
                \hat{Q} \cap C_d^{(r)} = \emptyset \right)
            \leq \frac{\eps^2}{64} \cdot \frac{4}{\eps^2}
            = \frac{1}{16} \,.
    \end{equation*}
    Together with \eqref{eq:hit-S_k} we find
    \begin{align*}
       \P\left(|\hat{Y}_{r,d}| \geq \frac{\eps}{2} \right)
       & = \P\left( |\hat{Y}_{r,d}| \geq \frac{\eps}{2} 
                    \;\middle|\;  \hat{Q} \cap C_d^{(r)} = \emptyset
            \right)
                \cdot\P\left( \hat{Q} \cap C_d^{(r)} = \emptyset \right) 
       \\
       & \qquad+\P\left( |\hat{Y}_{r,d}| \geq \frac{\eps}{2}
                        \;\middle|\;  \hat{Q} \cap C_d^{(r)} \neq \emptyset
                \right)
                    \cdot \P\left( \hat{Q} \cap C_d^{(r)} \neq \emptyset \right)
             \\
        &\leq \frac{1}{16}\cdot 1 + 1\cdot \frac{1}{16} =\frac{1}{8} \,.
    \end{align*}
    Based on this probability,
    a well known estimate on the median of $R$ independent repetitions, see e.g.~\cite[eq~(2.6)]{NP09}, gives
    \begin{equation*}
        \P\left(Z_d \geq \frac{\eps}{2}\right) < 2^{-(1+R/2)} \,.
    \end{equation*}

    In contrast, 
    assume that $d \in Q$ with heavy hitter $j \in J_d$.
    Similarly to~\eqref{eq:sketchmoment}, we estimate
    \begin{equation*}
        \expect \left[(\hat{Y}_{r,d} - \sigma_{rj} x_j)^2 \;\middle|\;
                \hat{Q} \cap C_d^{(r)} = \emptyset \right]
            \leq \|\vec{x}_{J_d \setminus\{j\}}\|_2^2 + \frac{\eps^2}{128}
            \leq \frac{|x_j|^2}{64} \,.
    \end{equation*}
    By Chebyshev's inequality we have
    \begin{equation*}
        \P\left(|\hat{Y}_{r,d} - \sigma_{rj} x_j| \geq \frac{|x_j|}{2} \;\middle|\;
                \hat{Q} \cap C_d^{(r)} = \emptyset \right)
            \leq \frac{|x_j|^2}{64} \cdot \frac{4}{|x_j|^2} =
            \frac{1}{16} \,,
    \end{equation*}
    and together with \eqref{eq:hit-S_k} we find
    \begin{align*}
        \P\left(|\hat{Y}_{r,d}| \leq \frac{\eps}{2} \right)
            \leq \P\left(|\hat{Y}_{r,d}| \leq \frac{|x_j|}{2} \right)
            \leq \P\left(|\hat{Y}_{r,d} - \sigma_{rj} x_j| \geq \frac{|x_j|}{2}\right)
            \leq \frac{1}{8} \,.
    \end{align*}
    The application of the median, again, yields
    \begin{equation*}
        \P\left(Z_d \leq \frac{\eps}{2}\right) < 2^{-(1+R/2)} \,.
    \end{equation*}

    We did not prove guarantees for buckets $d \notin Q$ with large noise $\|\vec{x}_{J_d}\|_2 > \frac{\eps}{8\sqrt{2}}$.
    However, there exist at most 
    $k = \left\lfloor \left(\frac{8\sqrt{2}}{\eps}\right)^p \right\rfloor$
    buckets with $\|\vec{x}_{J_d}\|_2 \geq \frac{\eps}{8\sqrt{2}}$, including all important buckets $d \in Q$, because
    \begin{equation*}
        1 = \|\vec{x}\|_{p}^p
            = \sum_{d=1}^{D} \|\vec{x}_{J_d}\|_{p}^p
            \geq \sum_{d=1}^{D} \|\vec{x}_{J_d}\|_{2}^p
            \geq \#\left\{d \in [D] \colon \|\vec{x}_{J_d}\|_{2} \geq \frac{\eps}{8\sqrt{2}}\right\}
                \cdot \left(\frac{\eps}{8\sqrt{2}}\right)^p .
    \end{equation*}
    Hence, if all bucket scores $Z_d$ are successful in the sense that for $\|\vec{x}_{J_d}\|_2 \leq \frac{\eps}{8\sqrt{2}}$ we have $Z_d < \frac{\eps}{2}$,
    and for heavy buckets $J_d$ with $d \in Q$
    we have $Z_d > \frac{\eps}{2}$, then 
    $Q \subseteq I_k$.
    The probability of the complementary event can be bounded
    by a union bound over the failure probabilities for up to $D$ buckets,
    namely,
    \begin{equation*}
        \P(Q \not\subseteq I_k)
            \leq \sum_{d \colon \left\|\vec{x}_{J_d}\right\|_2 \leq \frac{\eps}{8\sqrt{2}}
                    } \P\left(Z_d \geq \frac{\eps}{2}\right)
                + \sum_{d \in Q} \P\left(Z_d \leq \frac{\eps}{2}\right)
            \leq D \cdot 2^{-(1+R/2)} \leq \delta_1
    \end{equation*}
    by choice of $R$.
\end{proof}

\begin{remark} \label{rem:largeD}
    The hashing parameter $D$ will be chosen depending on the desired precision,
    oblivious of the problem size~$m$.
    If we end up with $D \geq m$,
    we can actually take the trivial hashing $\vec{h} = (1,\ldots,m)$ with $D = m$
    and $\select$ will identify all $\eps$-heavy coordinates in a non-adaptive manner.
    If, subsequently, these are observed directly,
    the entire approximation is technically still adaptive.
    The real power of adaptivity, however, does only unfold
    if buckets still contain many elements, see Section~\ref{sec:shrink}.
    In fact, an output $\wt{\vec{Z}}$ with entries
    \begin{equation*}
        \wt{Z}_i := \median \bigl\{ \sigma_{ri} \hat{Y}_{r,i} \colon r \in [R] \bigr\}
    \end{equation*}
    will already be a non-adaptive and unbiased approximation of $\vec{x}$
    which is precisely the idea of 
    the \emph{count sketch} algorithm~\cite{CCF04}.
\end{remark}

\begin{remark}[$p>2$] \label{rem:p>2nobuckets}
    Forming buckets by hashing and identifying the largest ones does not really help in the case of~$p>2$.
    Consider a vector with $k_1$ entries of value $\eps := (2k_1)^{-1/p}$
    and all other entries put to $(2m)^{-1/p}$, giving $\|\vec{x}\|_p \leq 1$
    but $\|\vec{x}\|_2 \geq 2^{-1/p} \cdot m^{\frac{1}{2} - \frac{1}{p}}$.
    The expected $\ell_2$-noise on an individual bucket is
    \mbox{$\expect \|\vec{x}_{B_j \setminus\{j\}}\|_2 \asymp m^{\frac{1}{2} - \frac{1}{p}} \cdot D^{-\frac{1}{2}}$} for $k_1 \ll D\leq m$,
    and for fully independent hashing it concentrates around the expectation.
    When aiming for bucket noise smaller than~$\eps/\gamma$,
    we would thus need to pick a hashing parameter $D \succeq m^{1-\frac{2}{p}} \cdot (\gamma/\eps)^2$, so the repetition parameter of Lemma~\ref{lem:bucketselect}
    would be $R \succeq \log D \succeq \log \eps^{-1} + \log m$.
    Even worse, when looking at the variance of a measurement $\hat{Y}_{r,d}$, see \eqref{eq:sketchmoment} in the proof of Lemma~\ref{lem:bucketselect},
    we get
    \begin{equation*}
        \expect \hat{Y}_{r,d}^2 \succeq \frac{\|\vec{x}\|_2^2}{G} \,.
    \end{equation*}
    For this to be smaller or equal $\eps^2$ with high probability,
    $G \succeq m^{1 - \frac{2}{p}} \cdot \eps^{-2}$ is needed.
    Hence, the cost for bucket selection alone is
    \begin{equation*}
        R \cdot G
            \succeq m^{1 - \frac{2}{p}} \cdot \left(\log \eps^{-1} + \log m\right) \cdot \eps^{-2} \,.
    \end{equation*}
    This is already the cost for known linear randomized approximation methods,
    see Remark~\ref{rem:nonada_p>2}.
\end{remark}

\subsection{Spotting a single heavy hitter by accelerated shrinking}
\label{sec:shrink}

This section is a review of~\cite[Sec~3.1]{IPW11}
with slight variations of the two levels
of the method given in \cite[Alg~3.1]{IPW11}.
Namely, at its fundamental level we study a shrinkage step $\shrink_{\vec{h}}$ which reduces a given candidate set with respect to a hashing~$\vec{h}$.
Combining shrinking steps in a sequential way, the adaptive one-sparse recovery scheme $\spot_{\alpha}$ is defined.

We start with the description of the shrinkage step.
Assume that for a vector $\vec{x} \in \R^m$ we are given a candidate set $S_0 \subseteq [m]$
for which we suspect that it contains an important coordinate.
Let $\vec{h} = (h_i)_{i =1}^m \in [D_0]^m$ be a fixed hash vector
with a new hashing parameter $D_0 \in \N$ further splitting up $S_0$.
The algorithm takes two random measurements
\begin{align}
    Y_1 &= L_1(\vec{x}) := \sum_{i \in S_0} \sigma_i x_i \,, \nonumber\\
    Y_2 &= L_2(\vec{x}) := \sum_{i \in S_0} \left(h_i - \frac{D_0+1}{2}\right) \sigma_i x_i \,,
        \label{eq:ShrinkInfo}
\end{align}
where $\sigma_i \sim \Uniform\{\pm 1\}$ are pairwise independent Rademacher coefficients.
The output is defined as the set
\begin{equation*}
    \shrink_{\vec{h}}(\vec{x},S_0) := \left\{i \in S_0 \colon h_i = \left\lceil \frac{Y_2}{Y_1} + \frac{D_0}{2} \right\rceil\right\} \subseteq S_0 \,.
\end{equation*}
In case of $Y_1 = 0$ we consider the set to be empty and the instance of the algorithm a failure.
The following result is based on~\cite[Lem~3.2]{IPW11}.
It provides a probabilistic guarantee that the algorithm $\shrink_{\vec{h}}$
yields a subset of $S_0$ that still contains the heavy hitter,
which also includes the prerequisite $Y_1 \neq 0$.

\begin{lemma}\label{lem:Shrink}
    Let $\alpha_0 \in (0,1)$, and let $\vec{h} = (h_i)_{i=1}^m$ be a hash vector with values $h_i \in [D_0]$ where $D_0 \in \N$.
    For $\vec{x} \in \R^m$ and $j \in S_0 \subseteq [m]$ assume the heavy hitter condition
    \begin{equation}\label{Eq:HHSpot}
        \|\vec{x}_{S_0 \setminus\{j\}}\|_2 \leq \frac{|x_j| \sqrt{\alpha_0}}{\sqrt{2} \, (2D_0 - 1)} \,.
    \end{equation}
    Then
    \begin{equation*}
        \P\bigl(j \in \shrink_{\vec{h}}(\vec{x},S_0)\bigr) \geq 1 - \alpha_0 \,.
    \end{equation*}
\end{lemma}
\begin{proof}
    The idea is that $Y_1$ and $Y_2$ together give approximate information on the hash-value $h_j$ of the heavy entry,
    namely $Y_1 \approx \sigma_j x_j$ and $Y_2 \approx \sigma_j \left(h_j - \frac{D_0+1}{2}\right) x_j$,
    hence $\frac{Y_2}{Y_1} \approx h_j - \frac{D_0+1}{2} =: a$.
    `Approximately equal' means $\frac{Y_2}{Y_1} - a \in \left(- \frac{1}{2}, \frac{1}{2}\right]$
    because in this case we have $\left\lceil \frac{Y_2}{Y_1} + \frac{D_0}{2} \right\rceil = h_j$
    which implies $j \in \shrink_{\vec{h}}(\vec{x},S_0)$.
    We define normalized deviations,
    \begin{align*}
        T_1 &:= \frac{Y_1}{\sigma_j x_j} - 1 = \frac{1}{\sigma_j x_j} \sum_{i \in S_0 \setminus\{j\}} \sigma_i x_i \,, \\
        T_2 &:= \frac{Y_2}{\sigma_j x_j} - a = \frac{1}{\sigma_j x_j} \sum_{i \in S_0 \setminus\{j\}}
            \underbrace{\left(h_i - \frac{D_0+1}{2}\right)}_{|\cdot| \,\leq\, (D_0-1)/2 \,=:\, b} \sigma_i x_i \,.
    \end{align*}
    On $\{|T_1| < 1\}$, using $|a| \leq \frac{D_0 - 1}{2} = b$, 
    we obtain the following upper bound for the deviation:
    \begin{equation*}
        \left|\frac{Y_2}{Y_1} - a\right|
            = \left|\frac{a + T_2}{1 + T_1} - a\right|
            = \left|\frac{T_2 - a T_1}{1 + T_1}\right|
            \leq \frac{|T_2| + |a T_1|}{1 - |T_1|}
            \leq \frac{|T_2| + b|T_1|}{1 - |T_1|}
            \,.
    \end{equation*}
    If $|T_1| < \epsilon := \frac{1}{1+4b} = \frac{1}{2D_0 - 1}$ and $|T_2| < b \epsilon$, we have
    \begin{equation*}
        \left|\frac{Y_2}{Y_1} - a\right|
            < \frac{2b \epsilon}{1-\epsilon}
            = \frac{1}{2} \,,
    \end{equation*}
    leading to the correct identification of $h_j$.

    We show that with probability at least $1-\alpha_0$
    these conditions on $T_1$ and $T_2$
    are met.
    With $\expect T_1^2 = \|\vec{x}\|_{S_0 \setminus \{j\}}^2/x_j^2$ and $\expect T_2^2 \leq b^2 \expect T_1^2$,
    Chebyshev's inequality yields 
    \begin{align*}
        \P (|T_1| \geq \epsilon) 
            &\leq \frac{\lVert \vec{x}_{S_0 \setminus \{j\}}\rVert_2^2
                        }{x_j^2\epsilon^2}
            = \left( 2D_0 - 1 \right)^2 
                \frac{\lVert \vec{x}_{S_0 \setminus \{j\}}\rVert_2^2}{x_j^2} \,, \\
        \P (|T_2| \geq b\epsilon) 
            &\leq \left( 2D_0 - 1 \right)^2 
                \frac{\lVert \vec{x}_{S_0 \setminus \{j\}}\rVert_2^2}{x_j^2} \,.
    \end{align*}
    A union bound gives
    \begin{equation*}
        \P (|T_1| \geq \epsilon \lor |T_2| \geq b \epsilon)
            \leq \P (|T_1| \geq \epsilon) + \P (|T_2| \geq b \epsilon)
         \leq 2\left( 2D_0 - 1 \right)^2 
            \frac{\lVert \vec{x}_{S_0 \setminus \{j\}}\rVert_2^2}{x_j^2} \,.
    \end{equation*}
    It is now easily checked that under the heavy hitter condition (\ref{Eq:HHSpot}) we have
    \begin{equation*}
        \P (|T_1| \geq \epsilon \lor |T_2| \geq b \epsilon) \leq \alpha_0.
    \end{equation*}
    This bounds the probability for misclassifying the hash value $h_j$.
\end{proof}

We generalize the second part of \cite[Alg~3.1]{IPW11} and its analysis \cite[Lem~3.3]{IPW11},
namely, we design the method for general uncertainty levels $\alpha$ without changing the overall structure of the algorithm.
Our approach thus differs from other improvements of the original one-sparse recovery
like the one in \cite[Lem~13]{NSWZ18}.
It is similar to the approach in~\cite[Thm~50]{LNW17}
but it requires a heavy hitter condition~\eqref{eq:heavy-hitter-cond-general}
with a larger, $\alpha$-dependent constant~$\gamma$ to start with
because we do not make use of a preconditioning step~\cite[Lem~49]{LNW17}.
It is possible to significantly reduce constants by employing preconditioning, but the asymptotic complexity and error rates would remain unchanged.
For the sake of simplicity we do not discuss this algorithmic feature.

By iterating the shrinkage algorithm, the method $\spot_\alpha$ will produce a nested sequence of sets $J = S_0 \supseteq S_1 \supseteq \ldots$ until we find a set $S_k$ that contains just one element (or is empty in case of failure).
This final set $S_k$ will then
be the output of $\spot_\alpha(\vec{x},J)$.
In any event, for $m \geq 2$, it will suffice to iterate the shrinkage algorithm at most
\begin{equation} \label{eq:k*}
    k^\ast = k^\ast(m) := 
        \max\left\{0, \left\lceil \log_{\frac{9}{8}} \frac{\log_2 m}{8} \right\rceil \right\}
\end{equation}
times with random hashings before identifying the most important coordinate \mbox{$j \in J$} with a final application of $\shrink_{\vec{h}^\ast}$ where we use a deterministic hash function $\vec{h}^\ast$
that is injective on the remaining candidate set $S_{k^\ast}$, compare \cite[Lem~3.1]{IPW11}.
This method is guaranteed to work with probability $1-\alpha \in (0,1)$
provided $j$ is a heavy hitter in the sense required by the lemma below.
For $k=0,\ldots,k^\ast - 1$ we use a sequence of independent hash vectors $\vec{H}^{(k)} = (H_i^{(k)})_{i \in J}$. 
For each $k$ the entries of $\vec{H}^{(k)}$ are pairwise independent with $H_i^{(k)} \sim \Uniform[D_k]$ and
\begin{equation} \label{eq:D_k}
    D_k = D_k(\alpha) := \left\lceil 2^{8 \cdot (9/8)^{k} + k + 2} \alpha^{-1} \right\rceil
        > 2^{8 \cdot (9/8)^k}
    \,.
\end{equation}
As long as $\#S_k > 1$, we perform the iteration
\begin{equation} \label{eq:spot_step_k}
    S_0 := J, \quad S_{k+1} := \shrink_{\vec{H}^{(k)}}(\vec{x},S_{k}) \,,
\end{equation}
and, if we still have $\# S_{k^\ast} > 1$, in a final step we yield the set
\begin{equation} \label{eq:spot}
    \spot_{\alpha}(\vec{x},J) := 
    \shrink_{\vec{h}^\ast}(\vec{x},S_{k^\ast}) \,,
\end{equation}
where $\vec{h}^\ast = \vec{h}^\ast(S_{k^\ast})$ shall be injective on $S_{k^\ast}$ and take values in $[\#S_{k^\ast}]$.
If we are successful, $\spot_{\alpha}(\vec{x},J) = \{j\}$ and we identify the heavy hitter.
In case of failure, $\spot_{\alpha}$ may return another one-element set or the empty set.
In total,
$\spot_{\alpha}(\vec{x},J)$ requires at most $2k^\ast + 2 \preceq \log \log m$ adaptive linear measurements.
Recall that, besides a hash vector $\vec{H}^{(k)}$, the procedure $\shrink_{\vec{H}^{(k)}}$ uses random Rademacher coefficients $\vec{\sigma} = (\sigma_i)_{i=1}^m$. In fact, we may use the same Rademacher coefficients for all shrinking steps because independence is not needed for these in the subsequent proof. (Still, $\vec{\sigma}$ needs to be independent of the hash vectors~$\vec{H}^{(k)}$.)

\begin{lemma}\label{lem:Spot}
    Let $\alpha \in (0,1)$ and assume that, for some $j \in J \subseteq [m]$, the vector $\vec{x} \in \R^m$ satisfies the
    heavy hitter condition
    \begin{equation}\label{Eq:HeavyHitterCond}
           \|\vec{x}_{J \setminus \{j\}}\|_2 \leq \frac{|x_j| \cdot \alpha^{3/2}}{2049 \sqrt{2}} \,.
    \end{equation}
    Then
    \begin{equation*}
      \P\Bigl( \spot_{\alpha}(\vec{x},J) = \{j\} \Bigr) \geq 1- \alpha \,.
    \end{equation*}
\end{lemma}

\begin{proof}
    Consider the following nested sequence of random sets that contain the heavy hitter~$j$, generated by the hash vectors~$\vec{H}^{(k)}$:
    \begin{equation*}
        \wt{S}_0 := J, \qquad
        \wt{S}_{k+1}
            := \left\{i \in \wt{S}_k \colon H^{(k)}_i = H^{(k)}_j \right\}
        \quad\text{for } k=0,\ldots,k^\ast - 1 \,.
    \end{equation*}
    The hashing parameters $D_k$ are chosen as
    \begin{equation*}
        D_k := \left\lceil \gamma_k^2 \alpha_k^{-1} \right\rceil,
        \qquad
        \text{where}\quad
        \alpha_k := \frac{\alpha}{2^{k+2}},\quad
        \gamma_k := 2^{4 \cdot b^k} \quad
        \text{with } b=\frac{9}{8},
    \end{equation*}
    hence, $\gamma_k \leq \sqrt{\alpha_k D_k}$
    and Lemma~\ref{lem:hash} implies that the event
    \begin{equation*}
        C_k := \left\{ \lVert \vec{x}_{\wt{S}_{k+1} \setminus \{ j\}} \rVert_2  
            \leq \frac{\lVert \vec{x}_{\wt{S}_k \setminus  \{j\}  }  \rVert_2}{\gamma_k} 
            \right\}
    \end{equation*}
    holds with probability at least $1 - \alpha_k$.
    Here we need independence of the hash vectors.
    We will consider the success of $\spot_{\alpha}$ conditioned on the event
    \begin{equation*}
        \widehat{C} := \bigcap_{k=0}^{k^\ast - 1} C_k \,,
    \end{equation*}
    for which a union bound gives the probability estimate
    \begin{equation*}
        P\bigl(\widehat{C}\bigr)
            \geq 1 - \sum_{k=0}^{k^\ast - 1} \alpha_k
            \geq 1 - \frac{\alpha}{4} \sum_{k=0}^{\infty} 2^{-k}
            = 1 - \frac{\alpha}{2} \,.
    \end{equation*}
    
    The following events relate to the success of the individual shrinking steps:
    \begin{align*}
        E_k &:= \left\{\shrink_{\vec{H}^{(k)}}(\vec{x},\widetilde{S}_k)
                    = \widetilde{S}_{k+1}
                \right\}
            \quad\text{for } k=0,\ldots,k^\ast - 1 \,, \\
        E_{k^\ast}
            &:= \left\{\shrink_{\vec{h}^{\ast}}(\vec{x},\widetilde{S}_{k^\ast})
                    = \{j\}
                \right\} .
    \end{align*}
    The intersection
    of these events implies success of $\spot_{\alpha}$:
    \begin{equation*}
        \widehat{E} := \bigcap_{k=0}^{k^\ast} E_k
            = \bigl\{\spot_{\alpha}(\vec{x},J) = \{j\}\bigr\} \,.
    \end{equation*}
    From Lemma~\ref{lem:Shrink} we know
    \begin{equation*}
        \P\left( E_k \;\middle|\;
                \bigl\| \vec{x}_{\wt{S}_k\setminus\{j\}}\bigr\|_{2}
                    \leq \frac{|x_j| \sqrt{\alpha_k}}{\sqrt{2} (2D_k - 1 )}
            \right)
            \geq 1 - \alpha_k \,.
    \end{equation*}
    We claim that successful hashing $\widehat{C}$ implies the required heavy hitter conditions
    \begin{equation} \label{eq:hatCimpliesHHC}
        \bigl\| \vec{x}_{\wt{S}_k\setminus\{j\}}\bigr\|_{2}
            \leq \frac{|x_j| \sqrt{\alpha_k}}{\sqrt{2} (2D_k - 1 )} \,,
            \qquad k=0,\ldots,k^\ast \,.
    \end{equation}
    Assuming that this is indeed the case, a union bound gives
    \begin{equation*}
        \P\bigl(\widehat{E} \,\big|\, \widehat{C}\bigr) 
            \geq 1 - \sum_{k=0}^{k^\ast} \alpha_k
            \geq 1 - \frac{\alpha}{2} \,,
    \end{equation*}
    hence, $\spot_{\alpha}$
    succeeds with the desired probability:
    \begin{equation*}
        \P\bigl(\widehat{E}\bigr)
        \geq \P\bigl(\widehat{E} \,\big|\, \widehat{C}\bigr) \P\bigl(\widehat{C}\bigr)
        \geq  (1 - \alpha/2)^2 > 1 - \alpha \,.
    \end{equation*}

    In the remainder of the proof we show \eqref{eq:hatCimpliesHHC} by means of induction provided the condition $\widehat{C}$ holds.
    Using $\gamma_k^2 \alpha_k^{-1} \geq 4\gamma_0^2 = 2^{10}$
    we see
    \begin{equation*}
        2D_k - 1 = 2 \lceil \gamma_k^2 \alpha_k^{-1} \rceil - 1
        \leq 2 \gamma_k^2 \alpha_k^{-1} + 1
        \leq \left(2 + 2^{-10}\right) \gamma_k^2 \alpha_k^{-1} \,,
    \end{equation*}
    and instead of \eqref{eq:hatCimpliesHHC} we will prove the even stronger
    heavy hitter conditions
    \begin{equation} \label{eq:strongerHHC}
        \left\|\vec{x}_{\wt{S}_k\setminus\{j\}}\right\|_{2}
            \leq \frac{|x_j| \cdot \alpha_k^{3/2}}{\sqrt{2} \left(2 + 2^{-10}\right) \gamma_k^2} \,,
            \qquad k=0,\ldots,k^\ast \,.
    \end{equation}
    For $k=0$ with $\widetilde{S}_0 = J$, $\alpha_0 = \alpha/4$, and $\gamma_0 = 2^4$, this condition reads
    \begin{equation*}
        \left\|\vec{x}_{J\setminus\{j\}}\right\|_{2}
            \leq \frac{|x_j| \cdot \alpha^{3/2}}{\sqrt{2} \left(2^{11} + 1\right)}
    \end{equation*}
    which is exactly what we required in the lemma.
    Assume that~\eqref{eq:strongerHHC} holds for a fixed $k < k^\ast$.
    Then, conditioned on $C_k$,
    we find \eqref{eq:strongerHHC} for $k+1$
    by the simple observation
    \begin{equation*}
        \frac{\alpha_k^{3/2}}{\gamma_k^3} \cdot \frac{\gamma_{k+1}^2}{\alpha_{k+1}^{3/2}}
            = 2^3 \cdot 2^{2 \cdot 4 \cdot b^{k+1} - 3 \cdot 4 \cdot b^k}
            = 2^3 \cdot 2^{4 (2b-3) b^k}
            \leq 
            2^{8b-9}
            = 1
    \end{equation*}
    where we use $b = \frac{9}{8}$, in particular $2b-3 < 0$ and $b^k \geq 1$.
\end{proof}

\begin{remark} \label{rem:spotiid}
    Assuming full independence of the Rademacher coefficients $\sigma_{i}$,
    we can give a slight improvement
    of Lemma~\ref{lem:Shrink} using Hoeffding's inequality,
    namely, as a heavy hitter condition for a shrink step we then only require
    \begin{equation*}
        \|\vec{x}_{S_k \setminus\{j\}}\|_2 \leq \frac{|x_j|}{(2D_k - 1)\sqrt{2 \log \frac{4}{\alpha_k}}} \,.
    \end{equation*}
    The heavy hitter condition in Lemma~\ref{lem:Spot} for the application of $\spot_\alpha$
    is then also slightly reduced in its $\alpha$-dependence:
    \begin{equation*}
        \|\vec{x}_{J \setminus \{j\}}\|_2
            \leq \frac{|x_j| \cdot \alpha}{1025 \sqrt{2\log \frac{16}{\alpha}}} \,.
    \end{equation*}
    A smaller heavy hitter constant $\gamma$ will permit a smaller hashing parameter $D \sim \gamma^p$ (for $1 \leq p \leq 2$)
    which is chosen such that all $\eps$-heavy coordinates
    are sorted into buckets on which they fulfil the heavy hitter condition.
    Eventually, however, in the information cost analysis
    only $\log D$ plays a role.
    On the other hand, adaptivity does start to help significantly if $D \ll m$,
    see Remark~\ref{rem:largeD}.
    Hence, if $D$ can be reduced,
    then the adaptive power of $\spot$ unfolds for smaller~$m$ already.
\end{remark}

\section{Adaptive uniform approximation}
\label{sec:AdaAppInf}

\subsection{The general algorithmic structure}
\label{sec:generalAlg}

For $m \in \N$, $\eps,\delta \in (0,1)$, $1 \leq p \leq 2$,
we will define a randomized algorithm
\begin{equation*}
    A_{p}^{\eps,\delta}: \Omega \times \R^m \to \R^m
\end{equation*}
which satisfies
\begin{equation*}
    \sup_{\substack{\vec{x} \in \R^m \\ \|\vec{x}\|_p \leq 1}}
        \P\left( \|A_{p}^{\eps,\delta}(\vec{x}) - \vec{x}\|_{\infty} > \eps\right) \leq \delta \,,
\end{equation*}
where we write $A_{p}^{\eps,\delta}(\vec{x}) = A_{p}^{\eps,\delta}(\cdot,\vec{x})$ for the random variable that is the output.
Recall that for $\|\vec{x}\|_p \leq 1$ there are at most $k_0 := \lfloor \eps^{-p} \rfloor$ coordinates~$j$ with $|x_j| \geq \eps$, so-called $\eps$-heavy coordinates.
The algorithm is defined in two preparatory stages and a final output stage:
\begin{enumerate}
    \item \emph{Finding important buckets:}
        This stage will return a collection of coordinate sets (buckets) that
        covers all $\eps$-heavy coordinates and isolates them in the sense that each of them fulfils a heavy-hitter condition~\eqref{eq:heavy-hitter-cond-general} on its respective bucket with a suitable heavy-hitter constant
        $\gamma = \gamma_p(\eps,\delta)$.
        We allow this stage to fail with probability at most $\frac{\delta}{2}$.
        In detail, we
        fix a hashing parameter \mbox{$D = D_p(\eps,\delta)$}
        and draw a hash vector $\vec{H} = (H_i)_{i=1}^m$
        with pairwise independent entries \mbox{$H_i \sim \Uniform[D]$}.
        By an implementation of $\select_{R,G,k}^{\vec{H}}(\vec{x})$
        with parameters \mbox{$R := R_{p}(D,\delta)$}, \mbox{$G := G_{p}(\eps)$}, 
        and \mbox{$k := k_{p}(\eps)$},
        we compute a $k$-element set of hash values,
        \begin{equation*}
            I_k := \select_{R,G,k}^{\vec{H}}(\vec{x}) \subseteq [D] \,,
        \end{equation*}
        which defines a family $(J_d)_{d \in I_k}$ of presumably important buckets $J_d \subseteq [m]$, see the definition~\eqref{eq:J_d} of buckets corresponding to the hash vector $\vec{H}$. The parameters will be chosen such that hashing fails to isolate the important coordinates with probability at most $\delta_0 := \frac{\delta}{4}$, and also $\select$ fails to identify the heavy buckets with probability at most $\delta_1 := \frac{\delta}{4}$.
    \item \emph{Spotting heavy hitters:}
        Let $\alpha := \delta/(2k_0)$. 
        With $\#I_k = k$ instances of $\spot_{\alpha}$
        we obtain a coordinate selection
        \begin{equation*}
            K := \bigcup_{d \in I_k} \spot_{\alpha}(\vec{x},J_d)
        \end{equation*}
        for which we hope that it contains all $\eps$-heavy coordinates.
        Conditioned on the success in the first stage, for each of the $\eps$-heavy coordinates (of which there are up to $k_0$)
        the corresponding instance of $\spot_\alpha$ only fails with probability~$\alpha$, so the failure probability of this stage is at most $\delta_2 := \frac{\delta}{2}$ by a union bound.
    \item \emph{Output:} Using $\#K \leq k$ direct entry queries,
        return $\vec{z} = \vec{x}_K^\ast =: A_{p}^{\eps,\delta}(\vec{x})$ defined via
        \begin{equation*}
            z_i = \begin{cases} x_i &\text{for } i \in K, \\
                                0 &\text{else.}
                    \end{cases}
        \end{equation*}
        If the previous stages were successful in the sense that $K$ contains all $\eps$-heavy coordinates, then
        \begin{equation*}
            \bigl\lVert A^{\eps, \delta}_{p}(\vec{x}) - \vec{x}  \bigr\rVert_{\infty} \leq \eps.
        \end{equation*}
\end{enumerate}

Let us point out that the algorithm is \emph{homogeneous} in the sense that for any factor $t \in \R \setminus \{0\}$ and any input $\vec{x} \in \R^m$ we have
\begin{equation*}
    A_{p}^{\eps,\delta}(t \vec{x}) = t \, A_{p}^{\eps,\delta}(\vec{x}) \,.
\end{equation*}
Moreover, the algorithm applies the same functionals to $\vec{x}$ and $t \vec{x}$ when realized with the same values of its random parameters.
Indeed, the adaption in $\select$ depends on the relative values of norm estimates, in $\spot_{\alpha}$, the subroutine $\shrink_{\vec{h}}$ uses the ratio of two functionals to decide.
For further context on homogeneous algorithms we refer to~\cite{KK24}.

\subsection{Parametric setup and information cost}
\label{sec:algparam}

In this section we provide the precise parametric setup for the algorithm described in Section~\ref{sec:generalAlg} together with an analysis of the information cost.
This then leads to several results on complexity and $n$-th minimal errors in terms of the standard Monte Carlo error~\eqref{eq:MCerr}, see Section~\ref{sec:e(n)}.

In view of Lemma~\ref{lem:Spot}, with $k_0 = \lfloor \eps^{-p} \rfloor$ and $\alpha = \frac{\delta}{2k_0}$, the heavy-hitter constant required for $\spot_\alpha$ to work is
\begin{equation*}
    \gamma = 2049\sqrt{2} \, \alpha^{-3/2}
        = 8196 \, k_0^{3/2} \,\delta^{-3/2}
        \simeq 8196 \, \eps^{-3p/2} \, \delta^{-3/2} .
\end{equation*}
According to Corollary~\ref{cor:hash}, we need to take the hashing parameter
\begin{equation*}
    D := \left\lceil \left(\frac{\gamma}{\eps}\right)^p \cdot \frac{k_0}{\delta_0} \right\rceil
    \simeq 4 \cdot 8196^p \, \eps^{-(3p^2/2 + 2p)} \, \delta^{-(3p/2 + 1)}
\end{equation*}
in order to guarantee that, after hashing,
all $\eps$-heavy coordinates fulfil a heavy-hitter condition on their respective bucket (with failure probability at most $\delta_0 = \frac{\delta}{4}$).
This may seem quite large (for $p=2$ we have $D\simeq 268\,697\,664 \, \eps^{-10} \, \delta^{-4}$), but since only $\log D$ will enter the error and complexity bounds, we may allow for it.
On the other hand, applying $\spot$ will only make sense if $D \ll m$,
see Remark~\ref{rem:largeD}.
We also have to admit that improvements of Remark~\ref{rem:spotiid} and preconditioning~\cite[Lem~49]{LNW17} have been ignored so far.
Eventually, $\select^{\vec{H}}_{R,G,k}$ significantly reduces the number of buckets we need to consider.
According to Lemma~\ref{lem:bucketselect}, we need to choose the algorithmic parameters as
\begin{align*}
    R &:= 2 \left\lceil \log_2 \frac{D}{2\delta_1} - \frac{1}{2} \right\rceil + 1
        \simeq 2 \log_2 \frac{2D}{\delta}
        \asymp \log \eps^{-1} + \log \delta^{-1}, \\
    k &:= \left\lfloor 2^{7p/2} \, \eps^{-p} \right\rfloor,
    \qquad G := 4k \,,
\end{align*}
so that bucket selection detects all heavy buckets with failure probability at most $\delta_1 = \frac{\delta}{4}$.
Then only $k$ buckets remain for which we have to perform one-sparse recovery
(for $p=2$ we have $k \simeq 128 k_0$, so we might still recover many more entries than needed).
The information cost of the first stage is
\begin{equation*}
    n_1 = R \cdot G \asymp \left(\log \eps^{-1} + \log \delta^{-1}\right) \cdot \eps^{-p} 
    \,.
\end{equation*}
The information cost of the second stage,
resulting from $k$-times repeating $\spot_\alpha$,
is bounded above by
\begin{equation*}
    n_2 = k \cdot (2k^\ast(m) + 2)
        \asymp \left(\log \log m \right) \cdot \eps^{-p}
        \,.
\end{equation*}
The information cost of the third stage is bounded by $n_3 = k$.
This gives the order of the total cardinality $n = n_1+n_2+n_3$ of the algorithm,
as summarized in the following theorem.

\begin{theorem}[Uniform approximation] \label{thm:UniformUB}
    Let $1 \leq p \leq 2$, $m \in \N$, and $\eps,\delta \in (0,1)$.
    For the randomized algorithm
    $A_{p}^{\eps,\delta}$
    with parameters as above one has
    \begin{equation*}
        \sup_{\substack{\vec{x} \in \R^m \\ \|\vec{x}\|_p \leq 1}}
            \P\left( \|A_{p}^{\eps,\delta}(\vec{x}) - \vec{x}\|_\infty > \eps\right) \leq \delta \,.
    \end{equation*}
    Moreover, one has
    \begin{equation*}
        \card(A_{p}^{\eps, \delta}) \preceq
                \left(\log\delta^{-1} + \log\eps^{-1}+\log\log m\right) \cdot \eps^{-p} .
    \end{equation*}
\end{theorem}

\subsection{Complexity for the standard Monte Carlo error}
\label{sec:e(n)}

Having found bounds for the probabilistic error, see Theorem~\ref{thm:UniformUB},
we readily obtain bounds for the expected error,
namely, we define
\begin{equation} \label{eq:A_pq^eps}
    A_{p}^{\eps} := A_{p}^{\bar{\eps},\bar{\delta}}
    \quad\text{with } \bar{\eps} = \bar{\delta} := \frac{\eps}{2} \,.
\end{equation}

\begin{corollary}\label{cor:CostExpectedError}
    Let $1 \leq p \leq 2$, $m \in \N$, and $\eps \in (0,1)$.
    For the randomized algorithm
    $A_{p}^{\eps}$ defined via \eqref{eq:A_pq^eps} we have the expected error
    \begin{equation*}
        \sup_{\substack{\vec{x} \in \R^m \\ \|\vec{x}\|_p \leq 1}}
            \expect \|A_{p}^{\eps}(\vec{x}) - \vec{x}\|_{\infty} \leq \eps \,.
    \end{equation*}
    Moreover,
    one has (for large $m$)
            \begin{equation*}
                \card(A_{p}^{\eps}) \preceq
                    \left(\log \eps^{-1} + \log\log m \right)
                        \cdot \eps^{-p} \,.
            \end{equation*}
\end{corollary}

\begin{proof}
    By construction of the output, we always have
    \begin{equation*}
        \|A_{p}^{\eps}(\vec{x}) - \vec{x}\|_{\infty} \leq \|\vec{x}\|_{\infty} \leq \|\vec{x}\|_p \,.
    \end{equation*}
    Hence it holds
    \begin{align*}  
        \expect \|A_{p}^{\eps}(\vec{x}) - \vec{x}\|_{\infty}
            &\leq \underbrace{\P\left( \|A_{p}^{\eps}(\vec{x}) - \vec{x}\|_{\infty} > \frac{\eps}{2}\right)}_{\leq \, \overline{\delta} \,=\, \eps/2}
                    \cdot \underbrace{\|\vec{x}\|_{\infty}}_{\leq 1} \\
            &\qquad    + \underbrace{\P\left( \|A_{p}^{\eps}(\vec{x}) - \vec{x}\|_{\infty} \leq \frac{\eps}{2}\right)}_{\leq 1} \cdot \, \frac{\eps}{2} \\
            &\leq \eps \,.
    \end{align*}
    The cost bound follows from plugging $\bar{\eps} = \bar{\delta} := \frac{\eps}{2}$
    into Theorem~\ref{thm:UniformUB}.
\end{proof}

By inversion of the cardinality bound we obtain upper bounds for the $n$-th minimal error.

\begin{theorem}\label{thm:errorrates}
    Let $1 \leq p \leq  2$
    and $m,n \in \N$.
    Then 
    \begin{equation*}
        e^{\ran}(n, \ell_p^m \embed \ell_{\infty}^m)
            \preceq 
            \min\left\{1,
            \left(\frac{\log n + \log \log m}{n}\right)^{\frac{1}{p}}\right\}.
    \end{equation*}
\end{theorem}
\begin{proof}
    For vectors $\|\vec{x}\|_p \leq 1$,
    we may achieve an expected $\ell_\infty$-error 
    below the threshold $\eps \in (0,1)$
    by employing the algorithm~$A_p^\eps$
    as long as
    cardinality of $A_p^\eps$ does not exceed the given budget~$n$.
    If no such $\eps \in (0,1)$ exists,
    we may resort to the zero algorithm $A_0$ that requires zero information and returns $A_0(\vec{x}) = \vec{0}$, resulting in the error $\|\vec{0} - \vec{x}\|_\infty \leq \|\vec{x}\|_p \leq 1$.

    Corollary~\ref{cor:CostExpectedError} states a cardinality bound for $\eps \in (0,1)$ with implicit constant $C = C_{p} > 0$, and for our purpose this bound should not exceed $n$:
    \begin{equation} \label{eq:cost<=n}
        \card(A_{p}^{\eps}) \leq
            C \cdot \left(\log \eps^{-1} +\log\log m\right)
                \cdot \eps^{-p}
            \stackrel{!}{\leq} n \,.
    \end{equation}
    Restricting to 
    $m \geq 16$ gives $\log \log m > 1$,
    without loss of generality we assume $C \geq 1$,
    and with this and $\log n \geq 0$ in mind we define
    \begin{equation} \label{eq:eps(n)}
        \eps := \left(C \cdot \frac{\log n + \log \log m}{n}\right)^{\frac{1}{p}} .
    \end{equation}
    If this definition leads to $\eps \geq 1$,
    we are in the case where the zero algorithm $A_0$ prevails.
    If, however, the definition leads to $\eps \in (0,1)$,
    observe that in particular it
    gives $\log \eps^{-1} \leq \log n$,
    hence, \eqref{eq:cost<=n} shows that the cardinality of $A_p^\eps$ is indeed
    bounded by~$n$.
\end{proof}

\begin{remark} \label{rem:nonada_p>2}
So far we did not cover uniform approximation for the weaker assumption $\|\vec{x}\|_p \leq 1$ with $p \in (2,\infty)$.
In this case, exploiting $\|\vec{x}\|_2 \leq m^{\frac{1}{2} - \frac{1}{p}} \|\vec{x}\|_p$ for $\vec{x} \in \R^m$ and homogeneity of the method presented,
we may simply define a method via the $\ell_2$-case,
\begin{equation*}
    A_p^{\eps} := A_2^{\eps'} 
    \qquad \text{where } \eps' := m^{-\left( \frac{1}{2} - \frac{1}{p}\right)}\eps \,.
\end{equation*}
With $\log (\eps')^{-1} \asymp \log \eps^{-1} + \log m$,
Corollary~\ref{cor:CostExpectedError} implies the cardinality bound
\begin{equation*}
    \card(A_p^{\eps}) \preceq \left(\log \eps^{-1} + \log m\right) \cdot \eps^{-2} \cdot m^{1 - \frac{2}{p}} \,,
\end{equation*}
which is the best we can hope for with any algorithm that employs bucket selection, see Remark~\ref{rem:p>2nobuckets}.
By inversion or by direct application of Theorem~\ref{thm:errorrates}, we find corresponding error rates for $p \in (2,\infty)$:
\begin{align*}
    e^{\ran}(n, \ell_p^m \embed \ell_{\infty}^m)
        &\leq \min\left\{1\,,\, 
                    m^{\frac{1}{2} - \frac{1}{p}} \cdot
                    e^{\ran}(n, \ell_2^m \embed \ell_{\infty}^m) \right\} \\
        &\preceq \min\left\{1\,,\,
                m^{\frac{1}{2} - \frac{1}{p}} \cdot \sqrt{\frac{\log n + \log \log m}{n}}
                \right\} \\
        &\asymp \min\left\{1,
                m^{\frac{1}{2} - \frac{1}{p}} \cdot \sqrt{\frac{\log m}{n}} \right\}.
\end{align*}
The last asymptotic equivalence follows from the fact that
$n \succeq m^{1 - \frac{2}{p}}$ if the asymptotic upper bound shall be smaller than~$1$, hence, together with $n < m$ (the relevant range of incomplete information), we have $\log n \asymp \log m$.
From~\cite[Lem~5]{Ma91}, however, see also~\cite[eq~(3.2.5)]{Ku17},
we already know a linear (thus non-adaptive) randomized method that achieves precisely this error rate.
\end{remark}

\begin{remark}[$q < \infty$]
    Algorithms for uniform approximation can also be applied
    to find $\ell_q$-approximations with $p < q < \infty$.
    The key observation is:
    If for a threshold $\eps \in (0,1)$ we have
    $K \supseteq K_{\eps} := \{j \in [m] \colon |x_j| \geq \eps\}$,
    then $\|\vec{x} - \vec{x}^\ast_K\|_{\infty} \leq \eps$.
    Trivially, $\|\vec{x} - \vec{x}^\ast_K\|_{\infty} \leq \|\vec{x}\|_p$,
    and assuming $\|\vec{x}\|_p \leq 1$, 
    by interpolation we have
    \begin{align*}
        \frac{1}{q} &= \frac{\lambda}{p} + \frac{1-\lambda}{\infty}
        \quad\text{with } \lambda = \frac{p}{q} \in (0,1)\,, \\
        \|\vec{x} - \vec{x}^\ast_K\|_{q}
            &\leq \|\vec{x} - \vec{x}^\ast_K\|_p^\lambda \cdot \|\vec{x} - \vec{x}^\ast_K\|_\infty^{1-\lambda}
            \leq 1^\lambda \cdot \eps^{1-\lambda}
            = \eps^{1-\frac{p}{q}}\,,
    \end{align*}
    see for instance \cite[Lem~2.4]{KNR19}.
    For $1 \leq p \leq 2$,
    Theorem~\ref{thm:errorrates} then leads to
    \begin{align*}
        e^{\ran}(n, \ell_p^m \embed \ell_q^m)
            &\leq \left[e^{\ran}(n, \ell_p^m \embed \ell_{\infty}^m)\right]^{1-\frac{p}{q}} \\
            &\preceq
                \left(\frac{\log n + \log \log m}{n}\right)^{\frac{1}{p} - \frac{1}{q}}
    \end{align*}
    In an upcoming paper~\cite{KW24c}, however, with the help of an algorithm that combines multiple levels of approximation, we will show the better error bound
    \begin{equation*}
        e^{\ran}(n, \ell_p^m \embed \ell_q^m)
            \preceq
                \left(\frac{\log \log \frac{m}{n}}{n}\right)^{\frac{1}{p} - \frac{1}{q}}
    \end{equation*}
    for $1 \leq p \leq 2$ and $p < q < \infty$.
\end{remark}

\section{Adaption versus non-adaption}
\label{sec:adagap}

We improve upon a result from~\cite{KNW24}.
Combining the upper bounds of Theorem~\ref{thm:errorrates}
with lower bounds for the non-adaptive Monte Carlo setting, we obtain the following gap between the error of adaptive and non-adaptive algorithms for a linear problem.
This gap is, to our knowledge, the largest known gap.

\begin{theorem} \label{thm:adagap}
Let $n \in \N$ and 
$m = m(n) := \bigl\lceil C e^{an^2}\bigr\rceil$ with the constants $C,a>0$ from \cite[Thm~2.7]{KNW24}.
Then
\begin{equation*}
    \frac{e^{\ran,\ada}(n,\ell_1^m \embed \ell_\infty^m)}
        {e^{\ran,\nonada}(n,\ell_1^m \embed \ell_\infty^m)}
    \preceq \frac{\log n}{n} \,.
\end{equation*}
\end{theorem}
\begin{proof}
    The lower bounds in \cite[Thm~2.7]{KNW24} show that,
    for the particular value of $m$ relative to~$n$,
    any non-adaptive algorithm will exhibit an error larger than~$\frac{1}{100}$.
    In contrast, the adaptive upper bound of Theorem~\ref{thm:errorrates},
    using the asymptotic equivalence $\log \log m \asymp \log n$ 
    in the particular setup,
    gives precisely the rate we state for the gap between adaption and non-adaption.
\end{proof}

The theorem above considers a different problem for every $n$.
We can also find 
a single linear operator~$S$
such that a similar gap up to logarithmic terms can be observed
for the $n$-th minimal error
as stated in the next corollary.

\begin{corollary} \label{cor:diagop}
    For every $\alpha > 0$ there exists a linear operator $S = S_{\alpha}$ such that (for large $n$)
    \begin{equation*}
        \frac{e^{\ran,\ada}(n,S)}{e^{\ran,\nonada}(n,S)}
            \preceq \frac{(\log n)^2 \cdot (\log \log n)^{1+\alpha}}{n} \,.
    \end{equation*}
\end{corollary}
\begin{proof}
Similarly to~\cite[Rem~3.4]{KNW24} we consider the following example
defined with suitable constants $C,a > 0$:
\begin{align*}
    &N_k := 2^k,
    \quad m_k := \bigl\lceil C e^{aN_k^2}\bigr\rceil
    \qquad\text{for } k \in \N\,,\\
    &S_{\alpha} \colon \ell_1 \to \ell_\infty, \; \vec{x} \mapsto \vec{z}
    \quad\text{where } z_i := \begin{cases}
        x_i &\text{for } 1 \leq i \leq m_1 \\
        \displaystyle \frac{x_i}{k \cdot (\log k)^{1+\alpha}}
            &\text{for } m_{k-1} < i \leq m_k,\; k\geq 2 \,.
    \end{cases}
\end{align*}
As in \cite{KNW24} we find non-adaptive lower bounds for $k \geq 2$:
\begin{align*}
    e^{\ran,\nonada}(N_k,S_{\alpha})
        &\geq \frac{1}{k (\log k)^{1+\alpha}}
            \cdot e^{\ran,\nonada}(N_k,\ell_1^{m_k} \embed \ell_\infty^{m_k}) \\
        &\geq \frac{1}{100 k (\log k)^{1+\alpha}}
            \asymp (\log N_k)^{-1} (\log \log N_k)^{-1-\alpha} \,.
\end{align*}
For the upper bounds we define `splitting' indices $l_j := 2^j$ for $j \in \N$,
and index sets $\mathfrak{m}_1 := [m_{l_1}]$, $\mathfrak{m}_j := [m_{l_j}] \setminus [m_{l_{j-1}}]$ for $j \geq 2$.
Given a budget of $N_k = 2^k$ pieces of information,
we spend at most $n_j$ 
measurements on approximating $\vec{x}_{\mathfrak{m}_j}$ for $j=1,\ldots,k$,
where the individual budgets form a decaying sequence:
\begin{equation*}
    n_j := \left\lfloor c_{\alpha}^{-1} j^{-1-\alpha} N_k \right\rfloor
        = \left\lfloor c_{\alpha}^{-1} j^{-1-\alpha} 2^k \right\rfloor,
    \qquad
    c_{\alpha} := \sum_{j=1}^\infty j^{-1-\alpha} < \infty \,.
\end{equation*}
Let $J_k$ be the last index~$j$ for which $n_j \geq 1$.
We find that $J_k \asymp 2^{k/(1+\alpha)}$, in particular, $J_k \geq k$ for large $k$.
Further, $n_j \asymp j^{-1-\alpha} \, 2^k$ for $j \leq J_k$.
For the approximation of $\vec{x}_{\mathfrak{m}_j}$ for $j \leq J_k$ we employ a homogeneous, adaptive algorithm $A_{n_j}$ as in Theorem~\ref{thm:errorrates}, while approximating the tail $\vec{x}_{\N \setminus [m_{(2^{J_k})}]}$ with zero,
resulting in the algorithm
\begin{equation*}
    A_k^{\alpha}(x)
        := S_{\alpha} \left(A_{n_1}(\vec{x}_{\mathfrak{m}_1}),
        \ldots, A_{n_{J_k}}(\vec{x}_{\mathfrak{m}_{J_k}}), 0, 0, \ldots \right) \,.
\end{equation*}
With $\log \log(\#\mathfrak{m}_j) \preceq \log N_{l_j} \preceq l_j = 2^j$,
we have the following error bounds on the approximation of the blocks:
\begin{align*}
    \expect\|\vec{x}_{\mathfrak{m}_j} - A_{n_j}(\vec{x}_{\mathfrak{m}_j})\|_\infty
        &\preceq \|\vec{x}_{\mathfrak{m}_j}\|_1 \cdot \frac{\log n_j + \log \log(\#\mathfrak{m}_j)}{n_j} \\
        &\preceq \|\vec{x}_{\mathfrak{m}_j}\|_1 \cdot \frac{(k + 2^j) \cdot j^{1+\alpha}}{2^k} \,.
\end{align*}
Applying the triangle inequality on the expected error of $A_k^\alpha$,
we obtain the following adaptive upper bound (using $J_k \geq k$ for large $k$):
\begin{align*}
    e^{\ran}&(N_k,S_{\alpha}) \preceq e(A_k^{\alpha}, S_{\alpha})\\
        &\preceq \sup_{\|\vec{x}\|_1 \leq 1} \left(\frac{\|\vec{x}_{\N \setminus [m_{(2^k)}]}\|_1}{2^{J_k} \cdot J_k^{1+\alpha}} + \sum_{j=1}^{J_k} \frac{\|\vec{x}_{\mathfrak{m}_j}\|_1}{\max\{2^{j-1} (j-1)^{1+\alpha} , 1\}} \cdot \frac{(k + 2^j) \cdot j^{1+\alpha}}{2^k} \right) \\
        &= \max\left\{\frac{1}{2^{J_k} \cdot J_k^{1+\alpha}} \,,\,
        \frac{1}{\max\{2^{j-1} (j-1)^{1+\alpha} , 1\}} \cdot \frac{(k + 2^j) \cdot j^{1+\alpha}}{2^k}\right\}_{j=1}^{J_k}  \\
        &\preceq \max\left\{ 2^{-k} \cdot k^{-1-\alpha},\;
                            2^{-k-j} \cdot k ,\;
                            2^{-k}\right\}_{j=1}^{J_k} 
        \leq k \cdot 2^{-k} \asymp \frac{\log N_k}{N_k} \,.
\end{align*}
Since $N_k$ is chosen in dyadic steps, we find asymptotic rates for 
general (large)~$n$.
Combining non-adaptive lower bounds and adaptive upper bounds, we prove the gap as stated.
\end{proof}

\begin{remark} \label{rem:largestgap}
    According to~\cite[Cor~14]{KNU24}, for linear problems
    (where we consider inputs from the unit ball of a normed space)
    the largest possible gap between adaptive and non-adaptive error rates is of order $n$ (up to logarithmic terms).
    Corollary~\ref{cor:diagop} gives such an extreme example.
    
    In~\cite[Rem.~3.4]{KNW24} examples of operators $S \colon \ell_1 \to \ell_2$ were found with a gap of order at least $\sqrt{n}$ (up to logarithmic terms) for adaptive vs.\ non-adaptive algorithms. Indeed, the gap cannot be larger if the target space is a Hilbert space, see~\cite[Cor~14]{KNU24}.
    As stated there, the same holds if the source space is a Hilbert space,
    and with our results for uniform approximation we can now also give matching examples with operators~$S \colon \ell_2 \to \ell_\infty$.

    Analogously to Theorem~\ref{thm:adagap}, with appropriate choice of $m = m(n)$, by Theorem~\ref{thm:errorrates} we find a sequence space embedding with
    \begin{equation*}
        \frac{e^{\ran,\ada}(n,\ell_2^m \embed \ell_\infty^m)}
            {e^{\ran,\nonada}(n,\ell_2^m \embed \ell_\infty^m)}
        \preceq \sqrt{\frac{\log n}{n}} \,.
    \end{equation*}
    For this result we use that lower bounds for $\ell_1^m \embed \ell_\infty^m$, see~\cite{KNW24}, also hold for $\ell_2^m \embed \ell_\infty^m$.
    In fact, lower bounds for the second case are even more easily obtained,
    compare \cite{PW12}, and we might choose
    $m(n) = \left\lceil C \cdot e^{an} \right\rceil$ instead of the larger value 
    \mbox{$m(n) = \left\lceil C \cdot e^{an^2} \right\rceil$}.
    In both cases, however, $\log \log m \asymp \log n$.

    Using a construction similar to the one in the proof of Corollary~\ref{cor:diagop},
    for all $\alpha > 0$ we find an infinite rank diagonal operator $S \colon \ell_2 \to \ell_\infty$ such that
    \begin{equation*}
        \frac{e^{\ran,\ada}(n,S)}
            {e^{\ran,\nonada}(n,S)}
        \preceq \frac{(\log n) \cdot (\log \log n)^{(1+\alpha)/2}}{\sqrt{n}} \,.
    \end{equation*}
    In detail, we define $S \colon \ell_2 \to \ell_\infty$, $\vec{x} \mapsto \vec{z}$, with
    \begin{equation*}
        z_i := \begin{cases}
                    x_i &\text{for } 1 \leq i \leq m_1 \,, \\
                    \displaystyle \frac{x_i}{\sqrt{k \cdot (\log k)^{1+\alpha}}} 
                        &\text{for } m_{k-1} < i \leq m_k,\; k\geq 2 \,.
                \end{cases}
    \end{equation*}
    The proof follows the lines of Corollary~\ref{cor:diagop},
    only that we have the square root of all asymptotic error estimates.
\end{remark}

\appendix

\bibliographystyle{amsplain}

\bibliography{lit}

\end{document}